\documentclass[a4paper,11pt]{article}
\usepackage{amssymb,amsmath,amsfonts,latexsym,amsthm}
\usepackage[mathscr]{eucal}
\usepackage{verbatim}
\usepackage{amscd}

\small\normalsize

\newtheorem{proposition}{Proposition}
\newtheorem{lemma}{Lemma}

\newtheorem{theorem}{Theorem}
\newtheorem{corollary}{Corollary}

\title{On symmetry group of Mollard code\thanks{
The  authors are supported by the Grant
the Russian Scientific Fund 14-11-00555. \newline
I. Yu. Mogilnykh and F. I. Solov'eva   are with the Sobolev Institute
of Mathematics and Novosibirsk State University, Novosibirsk,
Russia (emails:~\{ivmog,sol\}@math.nsc.ru).}
}

\author{ I. Yu. Mogilnykh, F. I. Solov'eva}

\begin{document}

\maketitle

\begin{abstract} For a pair of given binary perfect codes $C$
and $D$ of lengths %%
 $t$ and $m$ respectively, the Mollard construction outputs
a perfect code $M(C,D)$ of length $tm+t+m$, having subcodes $C^1$
and $D^2$, that are obtained from codewords of $C$ and $D$
respectively by adding appropriate number of zeros. %%
In this work we
 generalize of a result for symmetry
groups of
Vasil'ev codes \cite{AHS} and find the group
$Stab_{D^2}Sym(M(C,D))$. The result is preceded by and partially
based on a
 discussion of "linearity" of coordinate positions %%
 (points) in a nonlinear
perfect code (non-projective Steiner triple system respectively).
\end{abstract}

\section{Introduction}\label{Intro}
There are not so many results on the structure of automorphism group of perfect codes, even in binary case.
The investigation of automorphism group and symmetry group of any code is important since these groups are measures of symmetries of the code structure.
 In the present paper we propose two invariants for measuring
  the "linearity" of coordinate positions
 and points in a nonlinear
perfect code and non-projective Steiner triple system not %*
necessarily associated with perfect codes.
 The symmetry group of a perfect code is very closely related to the authomorphism group of its Steiner triple system. Beside of  the automorphism and symmetry groups, kernel, rank and Steiner triple system of a perfect code, the proposed in the paper these new invariants will be important tools in  further  research of structural properties of perfect codes.

The well-known result by Phelps \cite{Phelps} states that each finite group is isomorphic to the symmetry
group of a perfect binary code, whereas the result of Avgustinovich and Vasil'eva \cite{AV} established that the symmetry group of any perfect binary code of length $n$
is
isomorphic to the symmetry group of the subcode of all its codewords of weight $(n-1)/2$. However these results do not give the complete information on the structure of the symmetry and  automorphism
groups of perfect binary codes. The existence of  classes of perfect binary  codes with trivial
automorphism groups (nonsystematic and systematic) is considered in papers \cite{AS,M,Hed_tr_aut}.
It is well known \cite{MS} that the symmetry group $\mbox{Sym}(H)$ of the Hamming code $H$ of length $n$ is
isomorphic to the general linear group $GL(log(n+1), 2)$.
By the linearity of the Hamming code $H$
of length $n$, we have
$$|\mbox{Sym}(H)|=|\mbox{GL}(\log (n+1),2)|=n(n-1)(n-3)(n-7)\ldots(n-(n-1)/2).$$
The order of the automorphism group of an arbitrary nonlinear perfect
binary code was investigated by several authors, see the papers \cite{ST1,ST2,Mal_aut,Heden,Hed}.
The main definitions concerning this paper see in \cite{MS}.

\section{Notations and Definitions}\label{sec_Moll}

A collection $C$ of binary vectors
of length $n$ is called a {\it perfect} (1-perfect) code if any
binary vector is at distance 1 from exactly one codeword of $C$.
%%##
The perfect codes have length $n=2^m-1$, $2^{n-m}$
codewords and minimum distance 3. For every admissible  $n$ up to equivalence there is the
unique linear perfect code of length $n$, it is called the  {\it Hamming
code}.
A {\it Steiner triple system} is a collection of blocks
(subsets, called also triples) of size 3 of an $n$-element set, such that any
 unordered pair of distinct elements is exactly in one block.
  The set of codewords of weight 3 in a perfect code $C$, that contains the all-zero
 codeword is a Steiner triple system, which we denote ${\mathrm{STS}}(C)$. A Steiner triple system whose linear span is a Hamming code
 is called {\it projective}.

  With a Steiner triple system $S$ we
 associate a {\it Steiner quasigroup} $(P(S),\cdot)$ to be the point
set $P(S)$ of $S$ with a binary operation $\cdot$ such that:
$i\cdot j=k$, if $(i,j,k)$ is a triple of $S$ and $i\cdot i=i$. A {\it
Steiner loop} $(0\cup
 P(S),\star)$ with a binary operation $\star$ fulfills properties $i\star j=k$, if $(i,j,k)$ is a triple of $S$, $i\star i=0$ and $i\star
 0=i$.

Let $C$ and $D$ be two binary  one-error-correcting  codes of lengths $t$ and $m$ respectively.
Consider a representation for the Mollard construction \cite{Mol} for binary codes.

Consider the coordinate positions of the Mollard code $M(C,D)$ of
length $tm+t+m$ to be pairs $(r,s)$ from the set
$\{0,\ldots,t\}\times \{0,\ldots,m\}\setminus (0,0)$.

Let $f$ be an arbitrary function from $C$ to the set of binary vectors
of the vector space ${\bf F}_2^m$ of length $m$ and $p_1(z)$
and $p_2(z)$ be the generalized parity check functions:

$$p_1(z)=(\sum_{s=0}
^{m}z_{1,s},\ldots,\sum_{s=0}^{m}z_{t,s}),$$
 $$p_2(z)=(\sum_{r=0}^{
t}z_{r,1},\ldots,\sum_{r=0}^{t}z_{r,m}).$$
 The binary code $M(C,D)=\{z\in {\bf F}_2^{tm+t+m}: p_1(z)\in C, p_2(z)\in
 f(p_1(z))+D\}$ is called the Mollard code.
 In the case when $C$ and $D$ are perfect, the code $M(C,D)$ is perfect. Throughout the paper we consider the
 case when $f$ is the zero function, $C$ and $D$ are perfect codes, containing the all-zero words ${\bf 0}^t$ and ${\bf 0}^m$ respectively.

The Steiner triple system of $M(C,D)$ can be also defined using minimum weight codewords of the initial codes:
$$\mathrm{STS}(M(C,D))=\{x\in {\bf F}_2^{tm+t+m}: p_1(x)\in STS(C)\cup {\bf 0}^{t}, p_2(x)\in STS(D)\cup {\bf 0}^{m}\}\setminus \{{\bf 0}^{tm+t+m}\}.$$
We use the following convenient partition for the Steiner triple system of Mollard code

\begin{equation}\label{MSTSpart}\mathrm{STS}(M(C,D))=\bigcup_{k,p\in \{0,3\}} T_{kp} \end{equation}

 where
$$T_{00}=\{((r,0),(r,s), (0,s)): r\in \{1,\ldots,t\}, s\in \{1,\ldots,m\}\};$$
$$T_{33}=\{((r,s), (r',s'), (r'',s'')): (r, r', r'')\in \mathrm{STS}(C), (s, s', s'') \in \mathrm{STS}(D)\};$$
 $$T_{30}=\{ ((r,0),(r',s), (r'',s)): (r, r', r'')\in \mathrm{STS}(C), s\in \{0,\ldots,m\}\};$$
$$T_{03}=\{((r,s), (r,s'), (0,s'')): (s, s', s'')\in \mathrm{STS}(D), r\in \{0,\ldots,t\}\}.$$

%Let codes $S$ and $S'$ be vector representations
%of Steiner triple systems $S$ and $S'$ of orders $t$ and $m$ with
%all-zero words. Then the code $M(0^t\cup S,0^m\cup S')$ is the vector
%representation of Steiner triple system of order $tm+t+m$ with
%all-zero word.

Let $x$ and $y$ be codewords of $C$ and $D$ respectively. Denote by $x^{1}$
and $y^{2}$ codewords of $M(C,D)$ such that

    $$ (x^1)_{r0}=x_r, \mbox{for } r\in\{1,\ldots,t\}\mbox{ and } (y^2)_{0s}=y_s, \mbox{for } s\in\{1,\ldots,m\}$$
  with
zeros in all positions from ${\{0,\ldots,t\}\times\{1,\ldots,m\}}$
and $\{1,\ldots,t\}$ $\times\{0,\ldots,m\}$ respectively. Note that
$M(C,D)$ contains the codes $C$ and $D$ as the subcodes
$C^{1}=\{x^{1}: x \in C\}$ and $D^{2}=\{y^{2}: y \in D\}$
respectively.

  We also use a traditional representation of the Mollard code $M(C,D)$ using its subcodes $C^{1}$, $D^{2}$ and
    $\{e_{r,s}+e_{0,s}+e_{r,0}: r\in\{1,\ldots,t\}, s\in\{1,\ldots,m\}\}$,
    where $e_{r,s}$ is a vector of weight one with one in the  coordinate position $(r,s)$:

\begin{lemma}\label{RepMol} Given a vector $z\in M(C,D)$ there are unique codewords $x\in C$ and $y\in D$ such that

$$z={x}^{1}+{y}^2+\sum_{(r,s): z_{r,s}=1}(e_{r,s}+e_{0,s}+e_{r,0}).$$
\end{lemma}

Recall that {\it the dual} $C^{\perp}$ of a code $C$ is a
collection of all binary vectors $x$ such that
$\sum_{i=1,\ldots,n}x_ic_i=0$$\mbox{(mod 2)}$ for any codeword $c$
of $C$. For perfect codes $C$ and $D$, the dual of
the Mollard code $M(C,D)$ can be described in the following way:

\begin{equation}\label{Mollard_Dual}(M(C,D))^{\perp}=\{z:p_1(z)\in C^{\perp}, p_2(z)\in D^{\perp} \}.\end{equation}

{\it The rank} $\mathrm{rk}(C)$ of a code $C$ is defined to be the dimension of its linear span over ${\bf F}_2$. {\it The kernel} of the code is defined to be the subspace $\mathrm{Ker}(C)=\{x\in C: x+C=C\}$.
  The rank and kernel are important code invariants. Due to the construction, the Mollard code preserves many properties and characteristics of the initial codes $C$ and $D$, in particular, we have the iterative formulas for the size of kernel and rank:
$$dim(\mathrm{Ker}(M(C,D)))=dim(\mathrm{Ker}(C))+dim(\mathrm{Ker}(D))+tm$$
$$\mathrm{rk}(M(C,D))=\mathrm{rk}(C)+\mathrm{rk}(D).$$
%The previous formula was used in \cite{GMS} in solving the rank problem for propelinear perfect codes.

%\section{Symmetry group of a perfect code}% another title?

The {\it symmetry group} $\mathrm{Sym}(C)$ of a code $C$ (sometimes being called
the permutational automorphism group or full automorphism group \cite{MS})
is the subgroup of permutations on $n$ elements preserving the code setwise:
$$\mathrm{Sym}(C)=\{\pi \in S_n: \pi(C)=C \}.$$

%%##
The {\it automorphism group} of a Steiner triple system of order $n$ is the subgroup of permutations on $n$ elements preserving the collection of blocks of the system.

It is well-known that the symmetry group stabilizes the dual of the code, kernel \cite{PhelpsRifa} and its Steiner triple system:
\begin{equation} \label{SymKer}
\mathrm{Sym}(C)\leq \mathrm{Sym}(\mathrm{Ker}(C)),
\end{equation}

\begin{equation} \label{SymRot}
\mathrm{Sym}(C)\leq \mathrm{Sym}(C^{\perp}),
\end{equation}

\begin{equation} \label{SymSTSC}
\mathrm{Sym}(C) \leq \mathrm{Aut}(\mathrm{STS}(C)).
\end{equation}

By $Stab_{C}G$ and $Stab_{(C)}G$ of a code $C$ we denote the
setwise and codeword-wise stabilizers of the set $C$ by the group
$G$ acting on a code $C^{\prime}$, $C\subseteq C^{\prime}$. Let
$C$ be a perfect subcode of $C'$ on the nonzero coordinates $N(C)$. We have the following obvious statement.

\begin{proposition}
We have that for any perfect code $C$ with a perfect subcode $C'$ on coordinates $N(C)$:
$$Stab_{N(C)}Sym(C')=Stab_{C}Sym(C'), \,\, Stab_{(N(C))}Sym(C')=Stab_{(C)}Sym(C').$$
\end{proposition}

\section{Fundamental partition}\label{Fund_sec}

Given a perfect code $C$ of length $n$, one might
define {\it the fundamental partition} associated with $C$
\cite{AHS2}, \cite{PhelpsRifa} to be the partition of the coordinate set
$\{1,\ldots,n\}$ into subsets $I_0(C),\ldots,I_{2^{n-\mathrm{rk}(C)}-1}(C)$
such that for each $j\in\{0,\ldots,2^{n-\mathrm{rk}(C)}-1\}$ any codeword of
the dual code $C^{\perp}$ has the same values for coordinates with
indices from $I_j(C)$ \cite{AHS2}. By $I_0(C)$ we agree to denote the set of coordinates $\{i:x_i=0 \,\, \mbox{ for \,\, all } x\in C^{\perp}\}$
 which is of size $(n+1)/2^{n-\mathrm{rk}(C)}-1$ while $|I_j(C)|=(n+1)/2^{n-\mathrm{rk}(C)}$ for nonzero $j$. For the proof of the main result we
 essentially need the following fact:

\begin{lemma}\cite{PhelpsRifa}\label{fundPart}
Let $I_0(C),\ldots,I_{2^{n-\mathrm{rk}(C)}-1}(C)$ be the fundamental
partition associated with a perfect code $C$, $\pi \in Sym(C)$. Then
$$\pi(I_0(C))=I_0(C),$$
$$\mbox{for any } j\in \{1,\ldots,2^{n-\mathrm{rk}(C)}-1\} \mbox{ there is }j'\mbox{such that } \pi(I_j(C))=I_{j'}(C).$$
 \end{lemma}

 Avgustinovich at al.  \cite{AHS} considered a fundamental partition to show that any perfect code of rank $n-log_2(n+1)+2$
is obtained by Phelps construction. Utilizing a fundamental partition Heden in \cite{Heden} established an upper bound on the size of the symmetry group of a perfect code as a function of the rank of the code. In Section 5 we apply the idea of the work \cite{Heden} to prove our result on the symmetry group of a Mollard code.

From the description (\ref{Mollard_Dual}) of $(M(C,D))^{\perp}$ we obtain the following representation for the fundamental partition associated with the Mollard code $M(C,D)$:

$$I_{0}(M(C,D))=(I_0(C)\cup 0)\times (I_0(D)\cup 0)\setminus (0,0),$$
$$(I_0(C)\cup 0)\times I_{j'}(D),$$
$$I_{j}(C)\times (I_0(D)\cup 0),$$
$$I_j(C)\times I_{j'}(D), j=1,\ldots,t, j'=1,\ldots,m.$$

We also use the result of Heden (Lemma 8 of \cite{Heden}), which provides an inside view on the relationship of the code triples and the elements of
%%##
the fundamental partition:
\begin{lemma}\cite{Heden}\label{Heden_lemma}
Let $I_0(C),\ldots,I_{2^{n-\mathrm{rk}(C)}-1}(C)$ be the fundamental partition associated with $C$ of length $n$,
%%##
and  $(\{1,\ldots,n\},\star)$ be the Steiner loop associated with $STS(M(C,D))$. Then

1. for any $j\in\{0,\ldots,2^{n-\mathrm{rk}(C)}-1\}$, $r, r'\in I_j(C)$ we have that $r\star r'\in I_0(C)$;

2. for any $j, j'\in\{0,\ldots,2^{n-\mathrm{rk}(C)}-1\}$ there is a unique
%%##
$j\star' j'$ such that for $r\in I_j(C), r'\in I_{j'}(C)$ we have that
%%##
$r\star r'\in I_{j\star' j'}(C)$;

3. the set $\{I_0(C),\ldots,I_{2^{n-\mathrm{rk}(C)}-1}(C)\}$ with respect to the operation $\star'$ is an elementary abelian 2-group.
\end{lemma}

\section{Linear coordinates}
The topic of this section does not concern symmetries of perfect codes directly.
Here we discuss the idea of linear coordinates in a perfect code. We consider two characteristics for  coordinates of a perfect code or points of a Steiner triple system, which we use later for describing the symmetry groups of Mollard codes or the automorphism groups of Mollard Steiner triple systems. In this section we underline some of their properties and derive an important corollary that we use in the study of the symmetry group of a Mollard code.

 The set of the triples of a Steiner triple system
\begin{equation}\{(i,j,k), (i,a,b), (c,j,a), (c,k,b)\}\end{equation}
is called a {\it Pasch configuration} or, shortly, {\it Pasch}.

For a Steiner triple system $S$ on elements $\{1,\ldots,n\}$ and $i\in \{1,\ldots,n\}$,
%%##
$n\equiv 1, 3\pmod {6}$, define $\nu_i(S)$ to be the number of different Pasch configurations, incident to $i$, i.~e. such that there are two triples of the Pasch containing the point $i$.

For a  perfect code $C$ of length $n$ and a coordinate position $i$ we consider $\mu_i(C)$ to be the number of code
%%##
triples from $Ker(C)$ containing $i$:

$$\mu_i(C)=|\{x \in \mathrm{STS}(C)\cap \mathrm{Ker}(C): i\in supp(x)\}|.$$

Obviously, two coordinate positions $i, j$  of $S$ or $C$ are in different orbits by $Aut(S)$ or $Sym(C)$ respectively if $\nu_i(S)\neq \nu_j(S)$ or $\mu_i(C)\neq \mu_j(C)$ respectively. We say that a coordinate $i$ is $\mu$-{\it linear} for a code $C$ of length $n$ if $\mu_i(C)$ takes the maximal possible value, i.~e.
$(n-1)/2$. We say that a point $i\in \{1,\ldots,n\}$ is $\nu$-{\it linear} for a Steiner triple system $S$ of order $n$ if $\nu_i(S)$ takes the maximal possible value, i.~e. $(n-1)(n-3)/4$. By $Lin_\nu(S)$ and $Lin_\mu(C)$ denote the sets of $\nu$-linear coordinates of $S$ and $\mu$-linear coordinates of $C$ respectively.

\begin{lemma}\label{quasi_identities}
Let $<\{1,\ldots,n\},\cdot>$ be a quasigroup associated with a Steiner triple system $S$ of any order $n$.
 Then the following statements are equivalent:\\
\noindent
1. $l\in Lin_{\nu}(S)$;\\
\noindent
$2. \mbox{ for any distinct } s, s'\in \{1,\ldots,n\}, s, s'\neq l \mbox{ we have } (l\cdot s)\cdot(l \cdot s')=s\cdot s';$
$\\
\noindent
3. \mbox{ for any distinct } s, s'\in \{1,\ldots,n\}, s, s'\neq l \mbox{ we have} \,\, l\cdot (s \cdot s')= (l\cdot s)\cdot s'.$
\end{lemma}
\begin{proof}

A pair of different triples of $S$, containing $l\in Lin_{\nu}(S)$, e.g. $(l,s,l\cdot s)$
and $(l,s',l\cdot s')$ induces the following triples: $(s,s',s\cdot s')$ and $(l\cdot s, l\cdot s', s\cdot s')$.
 %%##A closer look at
 From the last block we have $(l\cdot s)\cdot(l \cdot s')=s\cdot s'$ for any different $s$ and $s'$
 %%##
  if and only if $l$ is $\nu$-linear.

 Now consider the triples $(l,s\cdot s', l\cdot(s\cdot s'))$ and $(l,s,l\cdot s)$. The coordinate $l$ is $\nu$-linear for $S$ iff there are triples $(s\cdot s', s, s')$ and $(l\cdot (s\cdot s'),l\cdot s, s')$ in $S$ for any $s$ and $s'$. Then we see that  $l\cdot (s \cdot s')= (l\cdot s)\cdot s'$ iff $l$ is $\nu$-linear.

\end{proof}

The second statement of the previous lemma implies that $0\cup Lin_{\nu}(S)$ is the {\it nucleus} of a Steiner loop,
associated with a Steiner triple system $S$, which, in particular implies  that if $S$ is nonprojective, then $|Lin_{\nu}(S)|< (n-1)/2$  (see, for example \cite{Voj}).

\begin{theorem}

1. Let $C$ be a perfect code. Then we have  $$Lin_{\mu}(C)\subseteq Lin_{\nu}(STS(C)).$$
\noindent
2. A subdesign of a Steiner triple system $S$ on $Lin_{\nu}(S)$ is a projective Steiner triple system.\\
\noindent
3. A subcode of a perfect code $C$ on the coordinates $Lin_{\mu}(C)$ is a Hamming code.
\end{theorem}
\begin{proof}
1. Let $i$ be $\mu$-linear coordinate. Then a pair of triples with the supports $\{i,s,i\cdot s\}$ and $\{i,r,i\cdot r\}$ incident to $i$ can be extended to a Pasch configuration. Indeed, since $i$ is $\mu$-linear, a codeword with the support $\{s,i\cdot s,r,i\cdot r\}$ is in $Ker(C)$,
which, being added with a code triple $\{r,s,r\cdot s\}$ gives a triple $\{i\cdot r,i\cdot s, r\cdot s\}$. Now it is easy to see that the four considered triples define a  Pasch.

2. Let $l$ and $l'$ be $\nu$-linear coordinates, $(l, l', l\cdot l')$ be a triple of $S$. We show that $l\cdot l'$ is $\nu$-linear.
Using the equalities from Lemma \ref{quasi_identities} for any distinct $s$ and $s'$ we have

$$((l\cdot l')\cdot s)\cdot((l\cdot l')\cdot s')=(l\cdot(l'\cdot s))\cdot(l\cdot (l'\cdot s'))=(l'\cdot s)\cdot(l'\cdot s')=s\cdot s',$$
which amounts to the fact that $l\cdot l'$ is $\nu$-linear.
The subdesign on $Lin_{\nu}(S)$ is a projective Steiner triple system, since it is known that a Steiner triple system with the maximum possible number of Pasch configurations is projective \cite{SW}.

3. Let $i$ and  $j$ be two $\mu$-linear coordinates. Then $i\cdot j$ is $\nu$-linear. For any fixed $r$ consider the following Pasch configuration:
\begin{eqnarray}
i\cdot j \, & (i\cdot j)\cdot r \, & \,\,\, r \nonumber \\
i\cdot j \, & j \, & \,\,\,  i\nonumber \, .\\
 \,  & r\cdot i \, & r\cdot i \nonumber
 \end{eqnarray}
A triple $(i\cdot j, (i\cdot j)\cdot r, r)$ is in Ker(C), since three remaining triples in the Pasch-configuration contain $i$ or $j$.
Now from the proven above we see that a subsystem of $STS(C)$ on the points $Lin_{\mu}(C)$ is projective, with all its triples being from $Ker(C)$. Therefore the subcode of $Ker(C)$ generated by these triples is a Hamming code.
\end{proof}

Finally, in the case of Mollard codes we have the following formulas for the number of triples in the kernel of Mollard code.
\begin{lemma}\cite{MS2}\label{lemma_muMollard}
Let $M(C,D)$ be a Mollard code obtained from
  perfect codes $C$ and $D$ of length $t$ and
$m$ respectively.
Then
\begin{enumerate}
    \item  $\mu_{(r,0)}(M(C,D))=\mu_{r}(C)(m+1)+m;$
\item $\mu_{(0,s)}(M(C,D))=\mu_{s}(D)(t+1)+t;$
\item $\mu_{(r,s)}(M(C,D))=1+2(\mu_{s}(D)+\mu_{r}(C)+\mu_{r}(C)\mu_{s}(D)).$
\end{enumerate}
\end{lemma}

\begin{corollary}\label{coroMu}
Let $M(C,D)$ be a Mollard code obtained from
  perfect codes $C$ and $D$ of length $t$ and
$m$ respectively. Then $\mu_{(r,s)}=\mu_{(r,0)}$ iff $s\in Lin_{\mu}(D)\cup 0$.

\end{corollary}

\section{The group $Stab_{D^{2}}Sym(STS(M(C,D)))$}

For a permutation $\pi$ on
the coordinate positions of the code $C$ (the code $D$), denote by ${\mathcal Dub}_1(\pi)$
 (${\mathcal Dub}_2(\pi)$ respectively) a permutation
 of coordinates of $M(C,D)$ such that
$${\mathcal Dub}_1(\pi)(r,s)=(\pi(r),s) \mbox{ if r is nonzero}, \,\, {\mathcal Dub}_1(\pi)(0,s)=(0,s) \mbox{ otherwise; }$$
$${\mathcal Dub}_2(\pi)(r,s)=(r,\pi(s)) \mbox{ if s is nonzero}, \,\, {\mathcal Dub}_2(\pi)(r,0)=(r,0) \mbox{ otherwise }$$
(see \cite{S2005}, \cite{BorgesMogilnykhRifaSoloveva}).
For a collection $\Pi$ of permutations we agree that ${\mathcal Dub}_i(\Pi)$ denotes $\{{\mathcal Dub}_i(\pi):\pi \in \Pi\}$, $i=1, 2$.
We have the following statement:

\begin{lemma}
\label{DubDub}
Let $C$ and $D$ be two perfect codes. Then

$$Stab_{C^{1}}Sym(M(C,D))\cap Stab_{D^{2}}Sym(M(C,D))=$$
$${\mathcal Dub}_1(Sym(C))\times{\mathcal Dub}_2(Sym(D)).$$

\end{lemma}
\begin{proof}
The inclusion of ${\mathcal Dub}_1(Sym(C))\times{\mathcal Dub}_2(Sym(D))$ into
the left hand side of the equality is obvious, see \cite{S2005}.
Let $\sigma$ be a permutation from $Stab_{C^{1}}Sym(M(C,D))\cap Stab_{D^{2}}Sym(M(C,D))$.
Consider $\pi$ to be a restriction of $\sigma$ on the nonzero coordinates of $C^1$, i.e. for any $r\in\{1,\ldots,t\}$ we have
$$\sigma(r,0)=(\pi(r),0), $$
which is equivalent to  $\sigma(c^1)=(\pi(c))^1$ for any $c\in C$.
We see that $(\pi(C))^1=\sigma(C^1)=C^1$ amounts to $\pi(C)=C$, so $\pi\in Sym(C)$.
Analogously we have that the restriction of $\sigma$ on the coordinates of $D^2$ is a permutation $\pi'\in Sym(D)$. Note that if $(r,0)$, $(0,s)$ are fixed by a permutation of coordinates from $M(C,D)$, then the coordinate $(r,s)$ is fixed, since there is a  codeword with the support $\{(r,0), (0,s), (r,s)\}$ in $M(C,D)$. This implies
that $\sigma$ must be equal to ${\mathcal Dub}_1(\pi){\mathcal Dub}_2(\pi')$.
\end{proof}

In this section we consider the structure of the setwise stabilizer $Stab_{D^{2}}Sym(STS(M(C,D)))$ (which we denote in what follows by $G$)
of the subcode $D^{2}$ in $Sym(M(C,D))$. The Mollard construction is a generalization of the Vasil'ev construction.
In \cite{AHS} the group of symmetries of Vasil'ev codes is investigated. In this section we obtain an extension of the result for Mollard codes.

Let ${\mathcal T}$ be a subgroup formed by the collection of symmetries $\tau$ of $G$ such that

\begin{equation}\label{PR1}
\mbox{ for any} \,\, r\in \{1,\ldots,t\}, \,\, \mbox{for any} \,\,  s\in \{0,\ldots,m\} \,\, \mbox{there exists} \,\,  s': \tau(r,s)=(r,s'),
\end{equation}

\begin{equation}\label{PR2}
\mbox{for any} \,\,  s\in \{1,\ldots,m\} \mbox{ we have }\tau(0,s)=(0,s).
\end{equation}

From Corollary \ref{coroMu} we obtain:
\begin{lemma}\label{RowOrbits}

 A symmetry $\tau\in {\mathcal T}$ setwise fixes the set of coordinates $\{(r,s):s\in0 \cup Lin_{\mu}(D)\}$ for any $r\in\{1,\ldots,t\}$.
\end{lemma}

\begin{proposition}\label{OGroup}
The group ${\mathcal T}$ is an elementary abelian 2-group.
\end{proposition}
\begin{proof}
We show that $\tau\in {\mathcal T}$ is necessarily of order not more than 2.
 Indeed, let $\tau(r,0)=(r,s)$, then, taking into account that $\tau(0,s)=(0,s)$, see (\ref{PR2}),
we have that a triple $\tau((r,0),(0,s),(r,s))=((r,s),(0,s),(r,s'))$ for some $s'$ must be in $STS(M(C,D))$. By (\ref{MSTSpart}) the triple
$((r,s),(0,s),(r,s'))$ is necessarily in $T_{00}$, so $s'=0$, i.e. $\tau(r,s)=(r,0)$. We see that $\tau^2$ fixes $(r,0)$ and $(0,s)$ for any $r\in \{1,\ldots,t\}, s\in \{1,\ldots,m\}$. Therefore, $\tau^2$ must fix $(r,s)$ for any $r\in \{1,\ldots,t\}, s\in \{1,\ldots,m\}$, because $\tau^2$ fixes elements $(r,0)$ and $(0,s)$ of the triple $((r,0), (0,s), (r,s))$. We have shown that $\tau^2$ is an identity. \end{proof}

We show that any element
of the group $G$ could be represented as a composition of the
following three symmetries: ${Dub}_2(\pi')$, for $\pi'\in Sym(D)$,
${Dub}_1(\pi)$, for $\pi\in Sym(C)$ and a symmetry $\tau\in
{\mathcal T}$. Here $\pi'\in Sym(D)$ is
%%##
the restriction of
$\sigma$ on the nonzero positions of the subcode $D^2$,
 $\pi\in Sym(C)$ is a permutation, induced by the action of $\sigma{Dub}_2(\pi'^{-1})$ on the subsets $r\times \{0,\ldots,m\}, r=1,\ldots,t$.

\begin{lemma}\label{reduce}
It is true that\\
 \noindent
1. $Stab_{(C^1)}G= {\mathcal Dub}_2(Sym(D))\vartriangleleft G;$\\
 \noindent
2.  $Stab_{(D^2)}G= \{{\mathcal Dub}_1(\pi)\tau:\pi \in Sym(C), \tau \in {\mathcal T}\}\vartriangleleft G;$\\
 \noindent
3. $G={\mathcal Dub}_2(Sym(D))\times \{{\mathcal Dub}_1(\pi)\tau:\pi \in Sym(C), \tau \in {\mathcal T}\}.$
\end{lemma}
\begin{proof}
Let $\sigma$ be from $G$. We have that $\sigma(D^2)=D^2$, so the restriction of $\sigma$ on $D^2$ is a permutation $\pi'\in Sym(D)$ (see the proof of Lemma \ref{DubDub}).

We now show that $\sigma'=\sigma {\mathcal Dub}_2(\pi'^{-1})$ acts on the following subsets of coordinates:
 $r\times \{0,\ldots,m\}, r\in\{1,\ldots,t\}$.
  For any $s\in \{1,\ldots,m\}$ let
$\sigma'(r,0)$ be $(r',s')$ and
$\sigma'(r,s)$ be $(r'',s'')$ for some $s', s''$ and nonzero
$r',r''$. Since $((r,0), (r,s),(0,s))$ is a triple of $M(C,D)$, so must be
$(\sigma'(r,0)$,$\sigma'(r,s)$,$\sigma'(0,s))$=
($(r',s')$,$(r'',s'')$,$(0,s)$).
From (\ref{MSTSpart}) the triple ($(r',s')$,$(r'',s'')$,$(0,s)$) is in $T_{00}$ or $T_{03}$ and both cases necessarily imply that $r'=r''$.

 So, there is a permutation $\pi$ of coordinate positions of the code $C$ such that

\begin{equation}\label{action_sigma} \sigma'(r\times \{0,\ldots,m\})=\pi(r)\times \{0,\ldots,m\},\end{equation}
  for $r\in\{1,\ldots,t\}$. The permutation $\pi$
is necessarily from $Sym(C)$ since $\sigma'$
should act as an element of $\mathrm{Sym}(C)$ on the first
%%%components
coordinates of the subcode $C^1$.
  For any $ x\in C$  we have that
$$p_1(\sigma'({x}^{1}))=p_1({\pi(x)}^1)=\pi(x)\in C,$$
which is true iff $\pi \in \mathrm{Sym}(C)$.

Therefore $\sigma$ is ${\mathcal Dub}_2(\pi') {\mathcal Dub}_1(\pi) \tau$ for some $\tau\in{\mathcal T}$. By definition of
%%##
${\mathcal T},$ see  (\ref{PR2}), the groups ${\mathcal T}$ and ${\mathcal Dub}_1(Sym(C))$ are subgroups of $Stab_{(D^2)}G$, \, ${\mathcal Dub}_2(Sym(D))$ is a subgroup of $Stab_{(C^1)}G$. Since a pointwise stabilizer of a group $G$ acting on a set is a normal subgroup of $G$, we obtain the required.
 \end{proof}

By Lemma \ref{reduce} we are now focused on the description of ${\mathcal T}$.
In the next lemma we use the idea similar to that of work \cite{Heden}.
\begin{lemma}\label{PropertiesO}
Let $(\{I_j(C),j=0,\ldots,2^{t-\mathrm{rk}(C)}-1\},\star')$ be a Steiner loop associated with the elements of the fundamental partition of the code $C$,
$(0\cup Lin_{\mu}(D),\star)$ be a subloop of Steiner loop associated with $STS(D)$ and
 $\tau$ be a symmetry of ${\mathcal T}$.
Then there is a group homomorphism

\noindent$\alpha:(\{I_j(C),j=0,\ldots,2^{t-\mathrm{rk}(C)}-1\},\star')\rightarrow (0\cup Lin_{\mu}(D),\star)$, such that

$$\tau(r,s)=(r,s\star\alpha(j)),$$ where $r\in I_j(C)$.

\end{lemma}

\begin{proof}

 For any $r\in \{1,\ldots,t\}$ define $l_r$ from the condition $\tau(r,0)=(r,l_r)$. By Lemma \ref{RowOrbits}, the coordinate $l_r$ must be in $0\cup Lin_{\mu}(D)$.
%%##
Consider a triple $((r,s), (0,s), (r,0))$. Since $\tau((r,s), (0,s), (r,0))$ is a triple of $M(C,D)$ and $\tau(r,0)=(r,l_r)$, $\tau(r,s)=(r,s')$, $\tau(0,s)=(0,s)$, we see that $p_2(e_{r,l_r}+e_{r,s'}+e_{0,s})=e_{l_r}+e_{s'}+e_s$
 must be in $D$, so either one of the elements
 %%##
 $l_r, s'$ is zero and the remaining is equal to $s$ or $(l_r,s,s')\in STS(D)$.
 This could be rewritten in
 %%##
 the form
 %%##
  $s'=s\star l_r$ and we have

\begin{equation}\label{PRR0}
\tau(r,s)=(r,s\star l_r).
\end{equation}

If $r$ is zero we set $l_0$ equal to 0 according to (\ref{PR2}).

We prove that $l_r=l_{r'}$, for $r, r'\in I_j(C)$. %, which in view of Lemma \ref{Heden_lemma} would give the required.
 Consider the restriction $\tau'$ of $\tau$ on the perfect subcode $M(C,D_{\mu})$ of the code $M(C,D)$, where $D_{\mu}$ is a linear subcode of $D$ on the positions $Lin_{\mu}(D)$. The restriction is correct, i.~e. $\tau'$ is a symmetry of $M(C,D_{\mu})$, since by Lemma \ref{RowOrbits} a symmetry $\tau$ fixes the set of the coordinates of $M(C,D_{\mu})$.
We have the following representation for the fundamental partition associated with $M(C,D_{\mu})$ (see Section \ref{Fund_sec}):

$$I_0(M(C,D_{\mu}))=I_0(C)\times 0,$$
$$I_j(C)\times s,\mbox{for all } s\in \{0,\ldots,m\},$$
$$(I_0(C)\cup 0)\times s\in \{1,\ldots,m\}.$$
 By (\ref{PR1}) and Lemma \ref{fundPart} we see that $\tau'$ fixes any element $(r,0)$ of $I_0(M(C,D_{\mu}))$, so
we have that $l_r$ is equal to $0$ for all $r\in I_0(C)$.

For any distinct $r$, $r'$, we have  $r\cdot r'=r\star r'$ and
\begin{equation}\label{Eq_lemma}\tau((r,0),(r',0),(r\star r',0))=((r,l_r),(r',l_{r'}),(r\star r',l_{r\star r'})).\end{equation}

If $r$, $r'\in I_j(C)$, $j\in\{0,\ldots,2^{t-\mathrm{rk}(C)}-1\}$ then
$r\star r'=r\cdot r'$ is in $I_0(C)$ (see Lemma \ref{Heden_lemma}) so $l_{r\star r'}=0$ and (\ref{Eq_lemma}) implies that $l_r=l_{r'}$. Therefore the action of $\tau$  can be presented as $(r,s)\rightarrow(r,s\star\alpha(j))$ if $r\in I_j(C)$ for some mapping $\alpha$ of  $\{I_j(C),j=0,\ldots,2^{t-\mathrm{rk}(C)}-1\}$ into $0\cup Lin_{\mu}(D)$.

Moreover, we have that $\alpha$ is an operation-preserving mapping.
By Lemma \ref{Heden_lemma} for any $j$, $j'$ there is a unique $j\star' j'$ such that
for $r\in I_j(C)$, $r' \in I_{j'}(C)$, $r\star r'$ is in $I_{j\star j'}(C)$. Because any triple $(\tau(r,0)$, $\tau(r',0)$, $\tau(r\star r',0))$
$=((r,l_r),(r',l_{r'}),(r\star r', l_{r \star r'})$
 must be a triple of $STS(M(C,D))$ we necessarily have that $\alpha(j)\star \alpha(j')=l_r\star l_{r'}=l_{r \star r'}=\alpha(j\star' j')$.
\end{proof}

%\begin{equation}\mbox{ for any }k\in ,r\in I_0(C)\mbox{ we have that } l_r=0\end{equation}

%\begin{equation}\mbox{ for any }k\in , r, r'\in I_k(C)\mbox{ we have that } l_r=l_{r'} \end{equation}
%So, any automorphism $\sigma'$ is uniquely defined by choosing $l_k$ for any $k\in 1,\ldots,$
Now from Lemma \ref{PropertiesO} we immediately obtain an evaluation for the order of ${\mathcal T}$.

\begin{corollary}\label{upperboundO} The order of ${\mathcal T}$ is not more than $(1+|Lin_{\mu}(D)|)^{t-\mathrm{rk}(C)}$.
\end{corollary}

%Now suppose $\sigma'(r,0)=(r,s)$. Then by Lemma \ref{RowOrbits} we see that $s$ is necessarily in $Lin_{\nu}(S)$. For any
%given $r$ let $l_r$ be such that $\sigma'(r,0)=(r,l_r)$.

%Furthermore, we have a convenient representation for the action of $\sigma'$:
%$$\sigma'(r,s)=(r,s\star l_r), \sigma'(0,s)=(0,s)$$
%where $\sigma'(r,0)=(r,l_r)$, for some $l_r\in Lin_{\mu(D)}$, $\star$ is the operation in the loop $L(STS(D))$. Indeed, a triple 5$(\sigma'(r,0),\sigma'(r,s),\sigma'(0,s))=((r,l_r),(r,s'),(0,s))$ is $M(C,D)$ iff $s'=s\star l_r$.

%First of all, we show that the statement of the theorem is true for perfect codes in the case when $D$ is a Hamming code.

For a codeword $u\in C$ and an element $l \in Lin_{\mu}(D)$, denote by $Ort_{l}(u)$ the permutation on the coordinates of $M(C,D)$ defined in the following way
%%##
$$Ort_{l}(u)(r,s)=(r,s \star l), \mbox{ for} \,\,\,  r \in supp(u), s\in \{0,\ldots,m\},$$
$$Ort_{l}(u)(r,s)=(r,s), \mbox{ otherwise},$$
where $\star$ is a binary operation in the Steiner loop associated with $STS(D)$.

%$$Ort^{1}_{\alpha}(u)(0,s)=(0,\alpha\cdot s) \mbox{ for any } s \in \{1,\ldots,m\}.$$
%$$Ort^{2}_{\alpha}(u)(r,s)=(r,\alpha \star s), \mbox{ for r }\in supp(u), s\in \{0,\ldots,m\}$$
%Ort^2 is not necessary
% Analogously $Ort^{2}_{\alpha}(u)(r,s)=(r\cdot \alpha, s), \mbox{for s}\in supp(u)\mbox{ and any r, } Ort^{2}_{\alpha}(u)(r,0)=(r\cdot \alpha, s)$ %for any $r$.
We agree that $Ort_{A}(U)$ denotes the collection of permutations $\{Ort_{l}(u):l\in A, u \in U\}$.

\begin{lemma}\label{lowerboundO}
Let $C$ and $D$ be perfect codes. Then
$<Ort_{Lin_{\mu}(D)}(C^{\perp})>\leq{\mathcal T}$ and $<Ort_{Lin_{\mu}(D)}(C^{\perp})>\cong Z_2^{(log_2(1+|Lin_{\mu}(C)|))^{t-\mathrm{rk}(C)}}$, here $t$ is length of the code $C$.

\end{lemma}
\begin{proof}
Let $u$ be an arbitrary nonzero vector from $C^\perp$, $l\in Lin_{\mu}(D)$, $z$ be an arbitrary
codeword of $M(C,D)$. We show that $Ort_{l}(u)$ is in $Sym(M(C,D))$.

By definition of $Ort_{l}(u)$, $p_1(Ort_{l}(u)(z))=p_1(z)$. Using Lemma \ref{RepMol}, $$z={x}^{1}+{y}^2+\sum_{(r,s): z_{r,s}=1}(e_{r,s}+e_{0,s}+e_{r,0})$$ for some $x\in C$ and $y\in D$.
Therefore we have the following equality:
 \begin{equation}\label{P2}p_2(Ort_{l}(u)(z))=p_2(Ort_{l}(u)({x}^{1}))+p_2(Ort_{l}(u)({y}^{2}))+p_2(\sum_{(r,s): z_{r,s}=1}(e_{r,s}+e_{0,s}+e_{r,0})).\end{equation}
We show that the righthanded side of (\ref{P2}) is a codeword of $D$.
 By definition of $Ort_{l}(u)$, we have that $Ort_{l}(u)({y}^{2})={y}^{2}$ and therefore $p_2(Ort_{l}(u)({y}^{2}))=y$. Since $u \in C^{\perp}$, there is the vector with the support $supp(u)\times \{0,\ldots,m\}$
 %%##is
  in $(M(C,D))^{\perp}$, see (\ref{Mollard_Dual}). Then the size of
$supp({x}^{1})\cap (supp(u)\times \{0,\ldots,m\})$ must be even. Since $supp({x}^{1})=supp(x)\times {\bf 0}^{tm}$, we have
that $$supp({x}^{1})\cap (supp(u)\times \{0,\ldots,m\})=(supp(x)\cap supp(u))\times {\bf 0}^{tm}.$$
 Since $Ort_{l}(u)({x}^{1})$ is obtained from ${x}^{1}$ by interchanging the subset of zero coordinates $supp(u)\times i$ and
 the coordinates from the subset $supp(u)\times {\bf 0}^{tm}$, we see that $p_2(Ort_{l}(u)({x}^{1}))$ is zero, since the block $supp(u)\times i$ contains even number of ones
 in $Ort_{l}(u)({x}^{1})$. So, $p_2(Ort_{l}(u)({x}^{1}))$ is zero.

  Now, by definition of $Ort_{l}(u)$  the triple
  $e_{r,s}+e_{0,s}+e_{r,0}$ is fixed by $Ort_{l}(u)$ for any $r \notin supp(u)$ and any $s\in\{1,\ldots,m\}$ and therefore $p_2(\sigma(e_{r,s}+e_{0,s}+e_{r,0}))={\bf 0}^m$. If $r$ is in $supp(u)$, then
    $\sigma( e_{r,s}+e_{0,s}+e_{r,0})=e_{r,s\star l}+ e_{0,s}+e_{r,l}$ and so we have that $$p_2(\sigma( e_{r,s}+e_{0,s}+e_{r,0}))=e_{s\star l}+e_s+e_l \in D$$ by definition of the operation $\star$.
Combining the obtained values for the righthand side of the equality (\ref{P2}) we have

$$p_2(Ort_{l}(u)(z))=y+\sum_{z_{r,s}=1, r\in supp(u)}(e_{s\star l}+e_s+e_l).$$

Any triple in the last sum is from $Ker(D)$, since it contains $l\in Lin_{\mu(D)}$, so we obtain that $p_2(Ort_{l}(u)(z))$ is in $D$.
Therefore $Ort_{l}(u)$ is a symmetry of  $M(C,D)$.

By Proposition \ref{OGroup}, we see that $<Ort_{Lin_{\nu}(D)}(C^{\perp})>$ is an elementary abelian 2-group.
A minimum set of generators for this group could be chosen to consist of symmetries $Ort_{l}(c)$, where $l$ runs through a minimal
generator set for the elementary abelian 2-group associated to the projective Steiner triple subsystem of $STS(D)$, defined on the points $Lin_{\mu}(D)$, and $c$ runs through a
set of generators of the code $C^{\perp}$. Therefore we have that $$<Ort_{Lin_{\mu}(D)}(C^{\perp})>\cong Z_2^{(log_2(1+|Lin_{\mu}(C)|))^{t-\mathrm{rk}(C)}}.$$
\end{proof}

By Corollary \ref{upperboundO} and Lemma \ref{lowerboundO} we have the description for $G$:

\begin{theorem}

%1. Let $C$ and $D$ be two perfect codes. Then
%$$G=({\mathcal Dub}_1(Sym(STS(C)))\rightthreetimes <Ort_{Lin_{\nu}(D)}(C^{\perp})>)\times{\mathcal Dub}_2(STS(Sym(D))).$$

% 2.
Let $C$ and $D$ be two reduced perfect codes. Then

$$G=({\mathcal Dub}_1(Sym(C))\rightthreetimes <Ort_{Lin_{\mu}(D)}(C^{\perp})>)\times{\mathcal Dub}_2(Sym(D)).$$

\end{theorem}

With a slightly shorter proof than that for the previous theorem we obtain the analogous result for Steiner triple systems:

\begin{theorem}
Let $S_1$ and $S_2$ be arbitrary two Steiner triple systems,
$M(S_1,S_2)$ be a Steiner triple system obtained from $S_1$ and
$S_2$ by applying
the Mollard construction. Then\\
$Stab_{S_2^2}Aut(M(S_1,S_2))=({Dub}_1(Aut(S_1))\rightthreetimes
<Ort_{Lin_{\nu}(S_2)}(S_1^{\perp})>)\times{Dub}_2(Aut(S_2)).$

\end{theorem}

In work \cite{MS2}  a class of Mollard
codes with symmetry groups, fixing $D^{2}$ fulfilling special
algebraic properties was obtained. By Theorem 2 we have a
description for the symmetry groups of this class.

\end{document}